\documentclass[12pt]{amsart}
 \usepackage{amsmath,amssymb,amscd,amsthm,esint}
\usepackage{graphics,amsmath,amssymb,amsthm,mathrsfs}

\begin{document}

\bibliographystyle{amsplain}

\newtheorem{theorem}{Theorem}[section]
\newtheorem{prop}[theorem]{Proposition}
\newtheorem{lemma}[theorem]{Lemma}
\newtheorem{definition}[theorem]{Definition}
\newtheorem{corollary}[theorem]{Corollary}
\newtheorem{example}[theorem]{Example}
\newtheorem{remark}[theorem]{Remark}
\newcommand{\ra}{\rightarrow}
\renewcommand{\theequation}
{\thesection.\arabic{equation}}

\def\xint#1{\mathchoice
  {\xxint\displaystyle\textstyle{#1}}%
  {\xxint\textstyle\scriptstyle{#1}}%
  {\xxint\scriptstyle\scriptscriptstyle{#1}}%
  {\xxint\scriptscriptstyle\scriptscriptstyle{#1}}%
  \!\int}
\def\xxint#1#2#3{{\setbox0=\hbox{$#1{#2#3}{\int}$}
  \vcenter{\hbox{$#2#3$}}\kern-.5\wd0}}
\def\ddashint{\xint=}
\def\dashint{\xint-}

\def \RD {{\mathbb R}^n}
\def \HS {H^1_{w}({\mathbb R}^n)}
\def\HSL  {H^1_{L,w}({\mathbb R}^n)}
\def \RN  {\mathbb{R}^n}
\def \HL {H^1_{L,w}({\mathbb R}^n)}
\def \HHL {\mathbb{H}^1_{L,w}({\mathbb R}^n)}
\def \LP {L^p_w}
\def\ls{\lesssim}
\def\varep{\varepsilon}
\def\brn{\mathbb{R}^n}

\title[$W^{1,p}$ estimates]
{Uniform $W^{1,p}$ estimates for systems of \\ linear elasticity  in a periodic medium}

\thanks{{\it Key words and phrase:} system of elasticity; homogenization; Lipschitz domain.}

\author[J. Geng, Z. Shen, and L. Song]
{Jun Geng \qquad Zhongwei Shen \qquad Liang Song}

\begin{abstract}
Let $\{\mathcal{L}_\varepsilon\}$ be a family of elliptic systems of linear elasticity
with rapidly oscillating periodic coefficients.
We obtain the uniform $W^{1,p}$ estimate $\|\nabla u_\varep \|_p \le C \| f\|_p$ 
in a Lipschitz domain $\Omega$
in $\mathbb{R}^n$
for solutions to the Dirichlet problem: $\mathcal{L}_\varep (u_\varep)=\text{div} (f)$
in $\Omega$ and $u_\varep=0$ on $\partial\Omega$,
where 
$|\frac{1}{p}-\frac12|<\frac{1}{2n}+\delta$ and $C$, $\delta>0$ are 
constants independent of $\varepsilon>0$.
The ranges are sharp for $n=2$ or $3$.
 \end{abstract}

 \medskip

\maketitle


\section {Introduction}
\setcounter{equation}{0}

The primary purpose of this paper is to study uniform $W^{1,p}$ estimates
for a family of elliptic systems of linear elasticity with rapidly oscillating coefficients.
Let $\Omega$ be a bounded Lipschitz domain in $\brn$, $n\geq 2$. 
Consider the Dirichlet problem
\begin{equation}\label{Dirichlet-problem}
\left\{
\aligned
\mathcal{L}_\varep (u_\varep) & =\text{div} (f) \quad \text{ in } \Omega,\\
 u_\varep & =0 \qquad\quad \text{ on } \partial\Omega,
\endaligned
\right.
\end{equation}
where
\begin{align}\label{e1.1}
\mathcal{L}_\varepsilon=-\frac{\partial}{\partial x_i}
\Big[a_{ij}^{\alpha\beta}\big(\frac{x}{\varepsilon}\big)
\frac{\partial}{\partial x_j}\Big]
=-{\rm div}\Big[A\big(\frac{x}{\varepsilon}\big)\nabla\Big], \qquad \varepsilon>0.
\end{align}
We will assume that the coefficient matrix  
$A(y)=(a_{ij}^{\alpha\beta}(y))$ is  real and satisfies
\begin{align}\label{e1.2}
a_{ij}^{\alpha\beta}(y)=a_{ji}^{\beta\alpha}(y)=
a_{\alpha j}^{i\beta}(y) \qquad  {\rm for} \  1\leq i,j,\alpha,\beta\leq n \ \  {\rm and} \ y\in \RD,
\end{align}
\begin{align}\label{e1.3}
\mu|\xi|^2\leq a_{ij}^{\alpha\beta}(y)\xi_i^{\alpha}\xi_j^\beta\leq \frac{1}{\mu}|\xi|^2 
            \qquad  {\rm for}  \ y\in \RD,
\end{align}
where $\mu$ is a positive constant and $\xi=(\xi_i^\alpha)$ is  any $n\times n$
{\it symmetric} 
matrix with real entries, and the periodicity condition
\begin{align}\label{e1.4}
A(y+z)=A(y)  \qquad {\rm for} \ y\in \RD \ {\rm and}\ z\in {\mathbb Z}^n.
\end{align}
We say $A\in \mathcal{M}(\mu,\lambda,\tau)$ if it satisfies
(\ref{e1.2}), (\ref{e1.3}), (\ref{e1.4}) and  the smoothness condition
\begin{align}\label{e1.5}
|A(x)-A(y)|\leq \tau|x-y|^{\lambda}  
\qquad {\rm for \  some} \ \lambda\in (0,1) \ {\rm and} \  \tau\geq 0.
\end{align}

The following is the main result of the paper.

\begin{theorem}\label{main-theorem-1}
Let $\Omega$ be a bounded Lipschitz domain in $\RD$.
Let $\mathcal{L}_\varep =-\text{\rm div} \big(A(x/\varep)\nabla\big)$
with $A\in \mathcal{M}(\mu,\lambda,\tau)$. 
Then for any $f\in L^p(\Omega)$ with $|\frac{1}{p}-\frac12|<\frac{1}{2n}+\delta$,
there exists a unique $u_\varep\in W^{1,p}_0(\Omega)$ such that
$\mathcal{L}_\varep (u_\varep)=\text{\rm div} (f)$ in $\Omega$.
Moreover, the solution $u_\varep$ satisfies
\begin{equation}\label{W-1-p-estimate}
\|\nabla u_\varep \|_{L^p(\Omega)} \le C_p\, \|f\|_{L^p(\Omega)},
\end{equation}
and constants $\delta>0$ and $C_p>0$ are independent of $\varep$.
\end{theorem}

We will also consider a family of general second-order elliptic systems
$\{-\text{div}\big(A(x/\varep)\nabla\big)\}$,
where $A(y)=(a_{ij}^{\alpha\beta}(y))$
with $1\le i,j\le n$ and $1\le\alpha, \beta\le m$.
We say $A\in \Lambda (\mu, \lambda, \tau)$ if it
satisfies (\ref{e1.4})-(\ref{e1.5}) and
the ellipticity condition (\ref{e1.3}) for any
$\xi=(\xi_i^\alpha)\in \mathbb{R}^{nm}$.
The symmetry condition $A=A^*$, i.e., $a_{ij}^{\alpha\beta}
=a_{ji}^{\beta\alpha}$, is also needed in the following theorem.

\begin{theorem}\label{main-theorem-2}
Let $\Omega$ be a bounded Lipschitz domain in $\RD$.
Let $\mathcal{L}_\varep =-\text{\rm div} \big(A(x/\varep)\nabla\big)$
with $A\in {\Lambda}(\mu,\lambda,\tau)$ and $A=A^*$. 
Then for any $f\in L^p(\Omega)$ with $|\frac{1}{p}-\frac12|<\frac{1}{2n} +\delta$,
there exists a unique $u_\varep\in W^{1,p}_0(\Omega)$ such that
$\mathcal{L}_\varep (u_\varep)=\text{\rm div} (f)$ in $\Omega$.
Moreover, the solution $u_\varep$ satisfies (\ref{W-1-p-estimate})
and constants $\delta>0$ and $C_p>0$ are independent of $\varep$.
\end{theorem}

Elliptic equations and systems with rapidly oscillating coefficients 
arise in the theory of homogenization (see e.g. \cite{bensoussan-1978, oleinik-1992}).
It is well known that as $\varep\to 0$,
the solution $u_\varep$ of (\ref{Dirichlet-problem})
converges to $u_0$ weakly in $W^{1,2}(\Omega)$ and strongly in $L^2(\Omega)$,
where $u_0\in W^{1,2}_0(\Omega)$ is the solution of the homogenized elliptic
system.
Uniform regularity estimates of $u_\varep$ are an important tool in the study of 
various convergence problems for $\mathcal{L}_\varep$.
We remark that if $\Omega$ is $C^{1,\alpha}$ and $A\in \Lambda(\mu, \lambda, \tau)$,
the uniform $W^{1,p}$ estimate (\ref{W-1-p-estimate})
was established in \cite{AL-1987, AL-1991}  for any $1<p<\infty$, without the symmetry condition
$A=A^*$ (see \cite{Kenig-Lin-Shen-1} for its extension to the Neumann boundary condition
with the symmetry condition).
It was pointed out in \cite{AL-1991} that the same approach also
gives estimate (\ref{W-1-p-estimate}) for $1<p<\infty$
if $\Omega$ is $C^{1,\alpha}$ 
and $A\in \mathcal{M}(\mu, \lambda,\tau)$.
We also mention that if $\Omega$ is Lipschitz and $m=1$,
the $W^{1,p}$ estimate (\ref{W-1-p-estimate}) was obtained in \cite{Shen-2008}
for $(4/3)-\delta<p<4+\delta$ and $n=2$, and for $(3/2)-\delta<p<3+\delta$ and $n\ge 3$.
The ranges of $p$'s in \cite{Shen-2008}
are known to be sharp (even for the Laplacian \cite{Jerison-Kenig-1995}).
It follows that the ranges of $p$'s in Theorems \ref{main-theorem-1}
and \ref{main-theorem-2} are sharp for $n=2$ or $3$.
The question of sharp ranges of $p$'s for which the $W^{1,p}$ estimate
holds in Lipschitz domains remains open in the case $n\ge 4$ (even for elliptic 
systems with constant coefficients).
We remark that in the non-periodic setting
the $W^{1,p}$ estimates 
for second-order elliptic equations and systems
have been studied
extensively in recent years. We refer the reader to \cite{Auscher,Byun-Wang-2004,Byun-2005,
Shen-2005,Krylov-2007,Byun-Wang-2008,Dong-2010} and their references
for various results on elliptic operators with nonsmooth coefficients in nonsmooth domains.

For a ball $B=B(x,r)$, we will use $tB$ to denote $B(x, tr)$.
Recall that $\Omega$ is a Lipschitz domain if there exists $r_0>0$ such that
for any $Q\in \partial\Omega$, there exists a Lipschitz function
$\psi: \mathbb{R}^{n-1}\to \mathbb{R}$ such that $\Omega\cap B(Q, 8r_0)$
is given by $\{ (x^\prime, x_n)\in \brn: \, x_n>\psi(x^\prime)\}\cap B(Q,8r_0)$,
after some possible translation and rotation of the coordinate system.
The proofs of Theorems \ref{main-theorem-1} and \ref{main-theorem-2}
rely on the following.

\begin{theorem}\label{main-theorem-3}
Let $\Omega$ be a bounded Lipschitz domain in $\brn$ and $q>2$.
Let $\mathcal{L}=-\text{\rm div} (A\nabla)$ be a second-order elliptic system
with $A=\big(a_{ij}^{\alpha\beta} (x)\big)$ and $1\le i, j\le n$,
$1\le \alpha, \beta\le m$. Suppose that
(1) $\|A\|_{L^\infty(\brn)}\le \mu^{-1}$; (2) for any $\phi \in W_0^{1,2}(\brn)$ and
some $\mu>0$,
\begin{equation}\label{quadratic-estimate}
\mu \int_{\brn} |\nabla \phi |^2\, dx
\le \int_{\brn} a_{ij}^{\alpha\beta} (x) \frac{\partial \phi^\alpha}{\partial x_i}
\cdot \frac{\partial \phi^\beta}{\partial x_j}\, dx;
\end{equation}
(3) for any $w\in W^{1,2}(3B\cap\Omega)$ with the property that
$\mathcal{L}(w)=0$ in $3B\cap\Omega$ and $w=0$ on $3B\cap\partial\Omega$ (if $3B\cap\partial\Omega
\neq \emptyset$), where either $3B\subset \Omega$ or
$B=B(y,r)$ with $y\in \partial\Omega$ and $0<r<r_0$, one has $|\nabla w|\in
L^q (B\cap \Omega)$ and 
\begin{equation}\label{reverse-Holder-inequality}
\left\{\frac{1}{|B\cap\Omega|}
\int_{B\cap\Omega} |\nabla w|^q\, dx \right\}^{1/q}
\le N
\left\{\frac{1}{|2B\cap\Omega|}
\int_{2B\cap\Omega} |\nabla w|^2\, dx \right\}^{1/2}.
\end{equation}
Then there exists  $\delta>0$, depending  only on $n$, $m$, $\mu$, $q$,
 $N$ and the Lipschitz character of $\Omega$,
such that for any $f\in L^p(\Omega)$ with $2<p<q+\delta$, the unique 
solution to $\mathcal{L}(u)=\text{\rm div} (f)$ in $W^{1,2}_0(\Omega)$ 
 satisfies $\|\nabla u\|_{L^p(\Omega)}
\le C_p \, \| f\|_{L^p(\Omega)}$, where
$C_p$ depends only on $n$, $m$, $\mu$, $p$, $q$, $N$ and the Lipschitz character of $\Omega$.
\end{theorem} 

Theorem \ref{main-theorem-3}, which is proved in Section 2,
follows by a real variable argument originated 
in \cite{Caffarelli-Peral} and further developed in \cite{Shen-2005}.
As an application of Theorem \ref{main-theorem-3}, in Section 3, we 
establish the $W^{1,p}$ estimate in the non-periodic setting for 
elliptic systems with VMO coefficients in Lipschitz domains.
Observe that by Lax-Milgram Theorem,
the conditions (1) and (2) in Theorem \ref{main-theorem-3}
give the existence and uniqueness of $W^{1,2}$ solutions for
any $f\in L^2(\Omega)$.
Clearly, the ellipticity condition in Theorem \ref{main-theorem-2}
implies the coercive estimate (\ref{quadratic-estimate}).
By the first Korn inequality this is also the case for Theorem \ref{main-theorem-1}.
Consequently, to prove Theorems \ref{main-theorem-1} and \ref{main-theorem-2},
as in \cite{Shen-2008},
it suffices to establish the weak reverse H\"older inequality (\ref{reverse-Holder-inequality})
with $q=p_n=\frac{2n}{n-1}$
for local $W^{1,2}$ solutions. We further note that under the assumption
$A\in \Lambda(\mu, \lambda, \tau)$ or $A\in \mathcal{M}(\mu, \lambda, \tau)$, 
it follows from \cite{AL-1987} that
\begin{equation}\label{interior-estimate}
\|\nabla u_\varep\|_{L^\infty(B)}
\le C \left\{ \frac{1}{|2B|}\int_{2B} |\nabla u_\varep|^2\, dx \right\}^{1/2},
\end{equation}
if $\mathcal{L}_\varep (u_\varep)=0$ in $3B$.
As a result we only need to establish (\ref{reverse-Holder-inequality})
for $w\in W^{1,2}(3B\cap\Omega)$ satisfying
$\mathcal{L}_\varep (w)=0$ in $3B\cap \Omega$ and $w=0$ on $3B\cap\partial\Omega$,
where $B=B(Q,r)$ for some $Q\in \partial\Omega$ and $0<r<cr_0$,
with constants $c$ and $N$ {\it independent} of the parameter $\varep>0$.
We will present two different approaches to this boundary reverse H\"older 
estimate.

The proof of Theorem \ref{main-theorem-2}, given in Section 4, uses the recently established
non-tangential maximal function estimates for the $L^p$ Dirichlet and 
regularity problems in \cite{Kenig-Shen-2} for 
some $p=q_0>2$, under the conditions $A\in \Lambda(\mu, \lambda, \tau)$
and $A=A^*$. Let $\rho(x)=\text{dist}(x, \partial\Omega)$.
To see (\ref{reverse-Holder-inequality}), the basic idea is to write
$$
\int_{B\cap\Omega} |\nabla w|^q\, dx
=\int_{B\cap\Omega} |\nabla w|^{q_0}\cdot |\nabla w|^{q-q_0}\, dx.
$$
and estimate $|\nabla w|^{q_0}$ by its (local) non-tangential maximal function and
$|\nabla w|^{q-q_0}$ by 
\begin{equation}\label{interior-estimate-1}
|\nabla w(x)|
\le {C}{\big[\rho (x)\big]^{-n/2}}
\left\{ \int_{2B\cap \Omega} |\nabla w|^2\, dy \right\}^{1/2}
\end{equation}
for any $x\in B\cap \Omega$, which follows from the interior estimate (\ref{interior-estimate}).
This gives (\ref{reverse-Holder-inequality}) for any $q<q_0+\frac{2}{n}$,
which can be used to improve the exponent of $\rho (x)$ in (\ref{interior-estimate-1}).
The desired estimate (\ref{reverse-Holder-inequality})
with $q=\frac{2n}{n-1}$ follows by an iteration argument.

In the case of elliptic systems of linear elasticity,
the non-tangential maximal function estimates used in the proof of Theorem \ref{main-theorem-2}
are not known.
To prove Theorem \ref{main-theorem-1}, we will instead adapt the approach used in 
\cite{Shen-2008} for single equations ($m=1$).
The idea is to reduce the estimate (\ref{reverse-Holder-inequality})
to a decay estimate of an integral of $|w|^q$ (not $|\nabla w|^q$) on  
a boundary layer and apply a compactness argument.
See Section 5 for details.

The summation convention is used throughout this paper.
Unless indicated otherwise $\Omega$ will always be a bounded Lipschitz 
domain in $\brn$. 
Finally, we will make no effort to distinguish vector-valued functions or function spaces
from their real-valued counterparts. This should be clear from the context.

\section{Proof of Theorem \ref{main-theorem-3}}
\setcounter{equation}{0}

By Lax-Milgram Theorem,
under the conditions (1) and (2) in Theorem \ref{main-theorem-3}, 
given any $f\in L^2(\Omega)$, the system $\mathcal{L}(u)=\text{div}(f)$
has a unique solution in $W^{1,2}_0(\Omega)$. Moreover, the solution satisfies
the estimate $\|\nabla u \|_{L^2(\Omega)} \le C \, \| f\|_{L^2(\Omega)}$,
where $C$ depends only on $\mu$.
Consider now the linear operator $T(f)=\nabla u$.
Clearly, $T$ is bounded on $L^2(\Omega)$.
To show that $T$ is bounded on $L^p(\Omega)$ for $2<p<q+\delta$,
we use the following theorem in \cite[Theorem 3.3]{Shen-2005}.

\begin{theorem}\label{theorem-shen}
Let $T$ be a bounded sublinear operator on $L^2(\Omega)$, where $\Omega$ is 
a bounded Lipschitz domain in $\brn$. Let $q>2$. Suppose that
there exists a constant $N>1$ such that
for any bounded measurable function $f$ with \text{\rm supp}$(f)\subset \Omega\setminus 3B$,
\begin{equation}\label{2.1.1}
\aligned
&\left\{\frac{1}{r^n}\int_{\Omega\cap B} |Tf|^q\, dx \right\}^{1/p}\\
& \le N \left\{\left(\frac{1}{r^n}
\int_{\Omega\cap 2 B} |Tf|^2\, dx\right)^{1/2}
+\sup_{B^\prime\supset B}
\left(\frac{1}{|B^\prime|}
\int_{B^\prime} |f|^q\, dx \right)^{1/q}\right\},
\endaligned
\end{equation}
where $B=B(x_0,r)$ is a ball with the property that
$0<r<c_0 r_0$ and either $x_0\in \partial\Omega$ or $B(x_0,3 r)\subset\Omega$.
Then $T$ is bounded on $L^p(\Omega)$ for any $2<p<q$.
\end{theorem}

It also follows from the proof of Theorem \ref{theorem-shen} in \cite{Shen-2005}
that if $\|T\|_{L^2(\Omega)\to L^2(\Omega)}\le C_0$, then 
$\| T\|_{L^p(\Omega)\to L^p(\Omega)}$ is bounded by a constant depending only 
on $p$, $q$, $N$, $C_0$, $c_0$ and the Lipschitz character of $\Omega$.
Therefore, to prove Theorem \ref{main-theorem-3} for $2<p<q$,
 it suffices to verify the condition (\ref{2.1.1}) with $T(f)=\nabla u$.
However, if supp$(f)\subset\Omega\setminus 3B$, one has
$\mathcal{L}(u)=0$ in $3B\cap \Omega$.
Thus the weak reverse H\"older inequality (\ref{reverse-Holder-inequality})
with exponent $q$ implies (\ref{2.1.1}) with the same exponent $q$
(without the supremum term in the right hand).
Finally, we observe that the weak reverse H\"older
condition (\ref{reverse-Holder-inequality})
is self-improving (see e.g. \cite{Giaquinta}).
That is, if $\mathcal{L}$ has the property (\ref{reverse-Holder-inequality})
for some $q=q_1>2$, then it has the property for some $q=q_1+\delta$, where
$\delta>0$ depends only on $n$, $q_1$, $N$ and the Lipschitz character of $\Omega$.
Consequently, by Theorem \ref{theorem-shen}, we obtain $\|\nabla u\|_{L^p(\Omega)}
\le C \, \| f\|_{L^p(\Omega)}$ for any $2<p<q +\delta$.
This completes the proof of Theorem \ref{main-theorem-3}.
\qed

\begin{remark}\label{duality-remark}
{\rm
Let $\mathcal{L}^*$ denote the adjoint of $\mathcal{L}$.
Suppose that $u, v \in W^{1,2}_0(\Omega)$ and 
$\mathcal{L}(u)=\text{div} (f)$ and $\mathcal{L}^* (v)=\text{div} (g)$
in $\Omega$ for some $f=(f_i^\alpha), g=(g_i^\alpha)\in L^2(\Omega)$. Then
\begin{equation}\label{duality}
\int_\Omega f_i^\alpha \cdot \frac{\partial v^\alpha}{\partial x_i}\, dx
=-\int_\Omega a_{ij}^{\alpha\beta} \frac{\partial u^\beta}{\partial x_j}
\cdot \frac{\partial v^\alpha}{\partial x_i}\, dx
=\int_\Omega g_i^\alpha \cdot \frac{\partial u^\alpha}{\partial x_i}\, dx.
\end{equation}
It follows from (\ref{duality}) by duality that
if the estimate $\|\nabla u\|_{L^p(\Omega)}
\le C \| f\|_{L^p(\Omega)}$ holds for any $f\in L^p(\Omega)$ and some $p>2$, then
$\|\nabla v\|_{L^q(\Omega)} \le C \| g\|_{L^q(\Omega)}$ for any $g\in L^2(\Omega)$,
where $q=p^\prime$.
By a density argument one may deduce that for any $g\in L^q(\Omega)$,
there exists $v\in W^{1,q}_0(\Omega)$ such that $\mathcal{L}^* (v)
=\text{div} (g)$ in $\Omega$ and $\|\nabla v\|_{L^q(\Omega)}\le C \| g\|_{L^q(\Omega)}$.
The duality argument above also gives the uniqueness of such solutions.
}
\end{remark}

\begin{remark}\label{Cacciopoli-remark}
{\rm
Under the conditions (1) and (2) in Theorem \ref{main-theorem-3}, the well known
Cacciopoli's inequality
\begin{equation}\label{Cacciopoli-inequality}
\int_{B\cap \Omega} |\nabla u|^2\, dx
\le \frac{C}{r^2}\int_{2B\cap\Omega} |u|^2\, dx
\end{equation}
holds for any $u\in W^{1,2}(3B\cap\Omega)$
 satisfying $\mathcal{L}(u)=0$ in $3B\cap\Omega$ and
$u=0$ in $3B\cap\partial\Omega$, where
$B=B(y,r)$ with $y\in \overline{\Omega}$
and $0<r<cr_0$. 
By Sobolev inequality this implies that
\begin{equation}\label{reverse-1}
\left\{\frac{1}{r^n}\int_{B\cap\Omega} |\nabla u|^2\, dx \right\}^{1/2}
\le C
\left\{\frac{1}{r^n}\int_{2B\cap\Omega} |\nabla u|^{q}\, dx \right\}^{1/q},
\end{equation}
for any $ 2n/(n+2)\le q<2$ if $n\ge 3$, and for any $1<q<2$ if $n=2$.
It follows that the weak reverse H\"older inequality (\ref{reverse-Holder-inequality}) holds
for some $q>2$ and $N>0$, which depend only on $n$, $m$, $\mu$ and the Lipschitz
character of $\Omega$ \cite{Giaquinta}.
}
\end{remark}

\section{$W^{1,p}$ estimates in the non-periodic setting}
\setcounter{equation}{0}

Following \cite{Shen-2005},
as an application of Theorem \ref{main-theorem-3}, we obtain the $W^{1,p}$
estimate in the non-periodic setting
for elliptic systems with VMO coefficients.
Recall that $A\in {\rm VMO}(\brn)$ if $\lim_{t\to 0} \omega(t)=0$, where
\begin{align}\label{e1.7}
\omega(t)=\sup\limits_{\stackrel{x\in \RD}{0<r<t}}\frac{1}{|B(x,r)|}
\int_{B(x,r)}\Big|A(y)-\frac{1}{|B(x,r)|}\int_{B(x,r)}A(z)dz \Big|dy.
\end{align}

\begin{theorem}\label{theorem-3.1}
Let $\Omega$ be a bounded Lipschitz domain in $\brn$.
Let $\mathcal{L}=-\text{\rm div}(A(x)\nabla)$ with $A(x)=(a_{ij}^{\alpha\beta} (x))$
and $1\le i, j\le n$, $1\le \alpha, \beta\le m$.
Suppose that (1) $\|A\|_{L^\infty(\mathbb{R}^n)} \le \mu^{-1}$; (2) estimate
(\ref{quadratic-estimate}) holds for any $\phi\in W^{1,2}_0(\mathbb{R}^n)$;
(3) $A=A^*$; and
(4) $A\in {\rm VMO}(\brn)$.
Then there exists $\delta>0$ such that
for any $f\in L^p(\Omega)$ with $|\frac{1}{p}-\frac12|< \frac{1}{2n} +\delta$, 
there exists a unique
$u\in W^{1,p}_0(\Omega)$ satisfying $\mathcal{L}(u)=\text{\rm div}(f)$ in 
$\Omega$. Moreover, the solution $u$ satisfies
$\|\nabla u\|_{L^p(\Omega)} \le C \| f\|_{L^p(\Omega)}$.
\end{theorem}

In view of Remark \ref{duality-remark} and the assumption $A^*=A$,
it suffices to prove Theorem \ref{theorem-3.1} for $2<p<p_n +\delta$, where
$p_n=\frac{2n}{n-1}$.
Furthermore, by Theorem \ref{main-theorem-3},
 we only need to establish the weak reverse H\"older
estimate in condition (3) in Theorem \ref{main-theorem-3} for $q=p_n$.
Note that under the condition $A\in {\rm VMO}(\brn)$, 
the estimate (\ref{reverse-Holder-inequality})
in the case $3B\subset \Omega$
is well known and in fact holds for any $2<q<\infty$.
As a result Theorem \ref{theorem-3.1} follows from the following.

\begin{theorem}\label{boundary-Holder-theorem}
Let $\mathcal{L}=-\text{\rm div}(A(x)\nabla)$ with $A(x)$ satisfying the
same conditions as in Theorem \ref{theorem-3.1}.
Suppose that $w\in W^{1,2}(3B\cap\Omega)$, $\mathcal{L}(w)=0$ in
$3B\cap\Omega$ and $w=0$ on $3B\cap\partial\Omega$, where $B=B(Q,r)$
with $Q\in \partial\Omega$ and $0<r<cr_0$.
Then $|\nabla w|\in L^{p_n}(B\cap\Omega)$ and
estimate (\ref{reverse-Holder-inequality}) holds.
\end{theorem}

Theorem \ref{boundary-Holder-theorem} is proved by a perturbation argument.
We first show that the desired estimate holds for elliptic systems 
$L=-\text{div} (\bar{A}\nabla)$ with constant coefficients
$\bar{A}=(\bar{a}_{ij}^{\alpha\beta})$
satisfying the Legendre-Hadamard ellipticity condition:
\begin{equation}\label{Hardamard-condition}
\mu |\xi|^2 |\eta|^2 \le \bar{a}_{ij}^{\alpha\beta} \xi_i\xi_j \eta^\alpha\eta^\beta
\le \mu^{-1} |\xi|^2|\eta|^2,
\end{equation}
for any $ \xi=(\xi_i)\in \brn, \, \eta=(\eta^\alpha)\in \mathbb{R}^m$.
It is known that the coercive estimate (\ref{quadratic-estimate}) and $\|A\|_\infty
<\infty$
imply the Legendre-Hadamard condition. In particular, 
if $\bar{a}_{ij}^{\alpha\beta} =\frac{1}{|E|}\int_E a_{ij}^{\alpha\beta} (x)\, dx$
and $(a_{ij}^{\alpha\beta})$ satisfies the conditions in Theorem \ref{theorem-3.1},
then $\bar{A}=(\bar{a}_{ij}^{\alpha\beta})$ satisfies
(\ref{Hardamard-condition}) and $(\bar{A})^* =\bar{A}$. 

\begin{lemma}\label{constant-coefficient-lemma}
Let $L=\text{\rm div}(\bar{A}\nabla)$ with constant coefficient
matrix $\bar{A}=(\bar{a}_{ij}^{\alpha\beta})$ satisfying (\ref{Hardamard-condition})
and $\bar{A}^*=\bar{A}$.
Suppose that $w\in W^{1,2}(3B\cap\Omega)$, $L(w)=0$ in $3B\cap\Omega$ and
$w=0$ on $3B\cap\partial\Omega$, where $B=B(y, r)$ with $y\in \overline{\Omega}$
and $0<r<cr_0$.
Then $|\nabla w|\in L^{p_n+\delta}(B\cap\Omega)$
and estimate (\ref{reverse-Holder-inequality}) holds for $q=p_n+\delta$,
where $\delta$ and $N$ in (\ref{reverse-Holder-inequality})
are positive constants depending only on $n$, $m$, $\mu$ and the Lipschitz character of $\Omega$.
\end{lemma}

\begin{proof}
Let $\psi: \mathbb{R}^{n-1}\to  \mathbb{R}$ be a Lipschitz function 
such that $\psi(0)=0$ and $\|\nabla\psi\|_\infty\le M$.  
For $r>0$, let
\begin{equation}\label{definition-of-Delta}
\aligned
&\Delta_r=\big\{(x',\psi(x'))\in \mathbb{R}^n: \  |x'|<r\big\},\\
&D_r=\big\{(x',t)\in \RD: |x'|<r \  {\rm and}\  \psi(x')<t<\psi(x')+(M+10n) r\big\}.
\endaligned
\end{equation}
Suppose that $ w\in W^{1,2}(D_{3r})$, 
${L}(w)=0$ in $D_{3r}$ and $w=0$ on $\Delta_{3r}$. We will show that
\begin{equation}\label{3.3.1}
\left\{\frac{1}{r^n}\int_{D_r} |\nabla w|^{p_n}\, dx \right\}^{1/p_n}
\le C \left\{ \frac{1}{r^n} \int_{D_{2r}} |\nabla w|^2\, dx \right\}^{1/2},
\end{equation}
where $C$ depends only on $n$, $m$, $\mu$ and $M$.
This, together with the interior estimates, yields (\ref{reverse-Holder-inequality})
for $q=p_n$
by a change of coordinates.
The case $q=p_n+\delta$ follows by the self-improvement property
of the weak reverse H\"older inequality.

To see (\ref{3.3.1}), we apply the $L^2$ estimates in \cite{Gao-1991}
 as well as square function estimates in \cite{Dahlberg-Kenig-Pipher-Verchota}
in the Lipschitz domain $D_{tr}$, where $t\in (1,2)$.
It follows that $\nabla w\in W^{1/2,2}(D_{tr})$ and by Sobolev imbedding,
 $|\nabla w|\in L^{p_n}(D_{tr})$.
Moreover, we obtain
\begin{equation}\label{3.3.2}
\aligned
\left\{\int_{D_{tr}} |\nabla w|^{p_n}\, dx \right\}^{1/p_n}
&\le C \left\{\int_{\partial D_{tr}} |\nabla w|^2\, d\sigma \right\}^{1/2}\\
& \le C \left\{\int_{\partial D_{tr}} |\nabla_{tan} w|^2\, d\sigma \right\}^{1/2},
\endaligned
\end{equation}
where $\nabla_{tan} w$ denotes the tangential gradient of $w$ on $\partial D_{tr}$
and $C$ depends only on $n$, $m$, $\mu$ and $M$.
Since $w=0$ on $\Delta_{3r}$, this gives
\begin{equation}\label{3.3.3}
\left\{ \int_{D_r} |\nabla w|^{p_n}\, dx \right\}^{2/p_n}
\le C\int_{\partial D_{tr}\setminus \Delta_{3r}} |\nabla w|^2\, d\sigma.
\end{equation}
Finally, we integrate both sides of (\ref{3.3.3}) with respect to $t$ over $(1,2)$
to obtain
\begin{equation}\label{3.3.4}
\left\{ \int_{D_r} |\nabla w|^{p_n}\, dx \right\}^{2/p_n}
\le \frac{C}{r} \int_{D_{2r}} |\nabla w|^2\, dx,
\end{equation}
from which estimate (\ref{3.3.1}) follows.
\end{proof}

\begin{lemma}\label{approx-lemma}
Let $\mathcal{L}=-\text{\rm div} (A(x)\nabla)$ with $A(x)$ satisfying the same conditions
as in Theorem \ref{theorem-3.1}.
Then there exist  a function $\eta(r)$ and some constants $N>0$
and $p>p_n$ with the following properties:

(1) $\lim\limits_{r\to 0} \eta(r)=0;$

(2) if $u\in W^{1,2}(3B\cap\Omega)$, $\mathcal{L}u=0$ in $3B\cap\Omega$ and $u=0$  
on $3B\cap \partial \Omega$, where $B=B(x_0,r)$ with $x_0\in \overline{\Omega}$ and $0<r<cr_0$, 
then there exists a function $v\in W^{1,p}(B\cap\Omega)$ such that

\begin{align}
\Big\{\frac{1}{r^n}\int_{B\cap\Omega}|\nabla(u-v)|^2
 \, dx \Big\}^{1/2}\leq \eta(r) 
\Big\{\frac{1}{r^n}\int_{3B\cap\Omega}|\nabla u|^2 \,dx \Big\}^{1/2},\label{3.4.1}\\
\Big\{\frac{1}{r^n}\int_{B\cap\Omega}|\nabla v|^p \, 
dx \Big\}^{1/p}\leq N
\Big\{\frac{1}{r^n}\int_{3B\cap\Omega}|\nabla u|^2 \,dx \Big\}^{1/2}. \label{3.4.2}
\end{align}
\end{lemma}

\begin{proof}
The proof is similar to that of Lemma 4.7 in \cite{Shen-2005}.
Suppose that $u$ satisfies the conditions of the lemma. 
We define the operator ${L} (w)=-D_ib_{ij}^{\alpha\beta}D_jw^\beta$, 
 where $D_i =\partial/\partial x_i$ and
$b_{ij}^{\alpha\beta}$ is a constant given by
\begin{align}
b_{ij}^{\alpha\beta}=\frac{1}{|B(x_0,3r)|}\int_{B(x_0,3r)}a_{ij}^{\alpha\beta}(x) \, dx.
\end{align}
Then $(b_{ij}^{\alpha\beta})$ satisfies the ellipticity condition (\ref{Hardamard-condition})
and $b_{ij}^{\alpha\beta}=b_{ji}^{\beta\alpha}$. 
Let $v$ be a weak solution of ${L} (v)=0$ in $2B\cap\Omega$
 such that $u-v\in W^{1,2}_0 (2B\cap\Omega)$. 
We will prove that $v$ satisfies estimates (\ref{3.4.1}) and (\ref{3.4.2}).

We first prove (\ref{3.4.1}). Note that
\begin{align}\label{3.4.3}
L(u-v)=(L-\mathcal{L})u=
-D_i(b_{ij}^{\alpha\beta}-a_{ij}^{\alpha\beta})D_ju^{\beta} \quad \text{ in } 2B\cap \Omega.
\end{align}
It follows that
\begin{equation}\label{3.4.4}
\aligned
\int_{2B\cap\Omega} & b_{ij}^{\alpha\beta} D_j(u-v)^{\beta}D_i(u-v)^{\alpha} \, dx\\
&\leq C \sum\limits_{i,j,\alpha,\beta}\int_{2B\cap\Omega}
|b_{ij}^{\alpha\beta}-a_{ij}^{\alpha\beta}||\nabla u||\nabla(u-v)| \, dx
\endaligned
\end{equation}
Since $(b_{ij}^{\alpha\beta})$ is a constant matrix satisfying
 the Legendre-Hadamard condition (\ref{Hardamard-condition}) 
and $u-v\in W^{1,2}_0 (2B\cap\Omega)$, we have
\begin{align}
\int_{2B\cap\Omega} & b_{ij}^{\alpha\beta} D_j(u-v)^{\beta}D_i(u-v)^{\alpha} 
\, dx\geq \mu \int_{2B\cap\Omega}|\nabla(u-v)|^2 \, dx,
\end{align}
which, together with (\ref{3.4.4}), gives
\begin{equation}\label{3.4.5}
\aligned
\Big\{\frac{1}{r^n} \int_{2B\cap\Omega}&|\nabla(u-v)|^2 \, dx\Big\}^{1/2}\\
&\leq C\sum\limits_{i,j,\alpha,\beta}\Big\{\frac{1}{r^n} 
\int_{2B\cap\Omega}|b_{ij}^{\alpha\beta}-a_{ij}^{\alpha\beta}|^2
|\nabla u|^2 \, dx\Big\}^{1/2}.
\endaligned
\end{equation}
By H\"older's inequality we have
\begin{equation}\label{3.4.6}
\aligned
&\Big\{\frac{1}{r^n} \int_{2B\cap\Omega}|\nabla(u-v)|^2 \, dx\Big\}^{1/2}\\
&\leq C\sum\limits_{i,j,\alpha,\beta}\Big\{\frac{1}{r^n} 
\int_{2B\cap\Omega}|b_{ij}^{\alpha\beta}-a_{ij}^{\alpha\beta}|^{2q_0'}\, 
dx\Big\}^{1/(2q_0')}\Big\{\frac{1}{r^n} \int_{2B\cap\Omega}
|\nabla u|^{2q_0}\, {\rm d}x\Big\}^{1/(2q_0)}\\
&\leq \eta(r)\Big\{\frac{1}{r^n} \int_{3B\cap\Omega}|\nabla u|^{2}\, {\rm d}x\Big\}^{1/2},
\endaligned
\end{equation}
where $q_0>1$ and we have used the weak reverse H\"older inequality
\begin{equation}\label{3.4.7}
\left\{ \frac{1}{r^n}\int_{2B\cap \Omega} |\nabla u|^{2q_0}\, dx\right\}^{1/(2q_0)}
\le C
\left\{ \frac{1}{r^n}\int_{3B\cap \Omega} |\nabla u|^2\, dx\right\}^{1/2}.
\end{equation}
Also, the function $\eta(r)$ above is defined by
\begin{align}
\eta(r)=C\sup\limits_{x_0\in \overline{\Omega}}
\sum\limits_{i,j,\alpha,\beta}\Big\{\frac{1}{r^n}\int_{B(x_0,2r)}
|b_{ij}^{\alpha\beta}-a_{ij}^{\alpha\beta}|^{2q_0'}\,{\rm d}x\Big\}^{1/(2q_0')}.
\end{align}
We recall  that the well known Cacciopoli's inequality holds under the conditions
(1) and (2) on $A(x)$ in Theorem \ref{theorem-3.1}.
As a consequence the weak reverse H\"older inequality (\ref{3.4.7})
 holds for some $q_0>1$ (see Remark \ref{Cacciopoli-remark}).
Since $a_{ij}^{\alpha\beta}\in {\rm VMO}$,
by the John--Nirenberg inequality, 
we have $\eta(r)\to 0$ as $r\to 0$. This completes the proof of (\ref{3.4.1}).

Finally, we note that since $L(v)=0$ in $3B\cap\Omega$ and $v=u=0$ on $3B\cap\partial\Omega$,
we may deduce from Lemma \ref{constant-coefficient-lemma} that
\begin{equation}\label{3.4.8}
\aligned
\Big\{\frac{1}{r^n}\int_{B\cap \Omega}\big|\nabla v\big|^p \, dx\Big\}^{1/p}
&\leq C \Big\{\frac{1}{r^n}\int_{2B\cap \Omega}\big|\nabla v\big|^2 \, dx\Big\}^{1/2}\\
&\leq C\Big\{\frac{1}{r^n}\int_{2B\cap\Omega}\big|\nabla u\big|^2 \, dx\Big\}^{1/2},\nonumber
\endaligned
\end{equation}
for some $p=p_n+\delta$, where the last inequality follows from (\ref{3.4.5}).
This gives (\ref{3.4.2}).
We point out that $\delta>0$ depends only on $n$, $m$, $\mu$ and the Lipschitz character of $\Omega$. 
\end{proof}

With Lemma \ref{approx-lemma} at our disposal,
Theorem \ref{boundary-Holder-theorem} follows from the following theorem, as in the proof of
Theorem C in \cite[p.192]{Shen-2005}. We omit the details.

\begin{theorem}\label{CP-theorem}
Let $f:E\to \mathbb{R}^m$ be a locally square integrable function, 
where $E$ is an open set of $\RD$. 
Let $p>2$. Suppose that there exist three constants $\varepsilon>0$ 
and $\alpha, N>1$ such that for every ball $B=B(x_0,r)$ 
with $\alpha B\subset E$, there exists a function $h=h_B\in L^p(B)$ with the properties:

\begin{align}
\Big\{\frac{1}{|B|}\int_B |f-h|^2\, dx \Big\}^{1/2}\leq 
\varepsilon \Big\{\frac{1}{|\alpha B|}\int_{\alpha B}|f|^2\, dx\Big\}^{1/2},\\
\Big\{\frac{1}{|B|}\int_B |h|^p\, dx \Big\}^{1/p}\leq N 
\Big\{\frac{1}{|\alpha B|}\int_{\alpha B}|f|^2\, dx\Big\}^{1/2}.
\end{align}
Then, if $2<q<p$ and $0<\varepsilon<\varepsilon_0=\varepsilon_0(n,m,p,q,\alpha,N)$, we have
\begin{align}
\Big\{ \frac{1}{|B|}\int_{B}|f|^q \, dx \Big\}^{1/q}
\leq  C\Big\{ \frac{1}{|\alpha B|}\int_{\alpha B}|f|^2 \, dx \Big\}^{1/2},
\end{align}
for any ball $B$ with $\alpha B\subset E$, where $C$ depends only on $n,m,p,q,\alpha$ and $N$.
\end{theorem}

We remark that Theorem \ref{CP-theorem}, which was stated in \cite[p.191]{Shen-2005}, 
was proved essentially in \cite{Caffarelli-Peral}.

\section{Proof of Theorem \ref{main-theorem-2}}
\setcounter{equation}{0}

In this section we give the proof of Theorem \ref{main-theorem-2}.
In view of Remark \ref{duality-remark} and Theorem \ref{main-theorem-3}, it suffices to show
that if $u_\varep \in  W^{1,2}(3B\cap\Omega)$ is a weak solution to
$\mathcal{L}_\varep (u_\varep)=0$ in $3B\cap \Omega$ and $u_\varep=0$
on $3B\cap\partial\Omega$, where $B=B(x_0,r)$ with $x_0\in \overline{\Omega}$ and
$0<r<cr_0$, then $|\nabla u_\varep|\in L^{p_n}(B\cap\Omega)$ and
estimate (\ref{reverse-Holder-inequality}) holds for 
$q=p_n=\frac{2n}{n-1}$ with a constant $N$
independent of $\varepsilon$. By the interior estimate (\ref{interior-estimate}) we may assume that
$B=B(Q,r)$ for some $Q\in \partial\Omega$.
Furthermore, by a change of coordinates, it is enough to show that
\begin{equation}\label{4.1}
\left\{\frac{1}{r^n}\int_{D_r} |\nabla u_\varep|^{p_n}\, dx\right\}^{1/p_n}
\le C
\left\{\frac{1}{r^n}\int_{D_{3r}} |\nabla u_\varep|^2\, dx\right\}^{1/2},
\end{equation}
whenever $u_\varep\in W^{1,2}(D_{3r})$ is a weak solution to
$\mathcal{L}_\varep (u_\varep)=0$ in $D_{3r}$ and $u_\varep =0$ on $\Delta_{3r}$.

Throughout this section we assume that $A\in \Lambda(\mu, \lambda, \tau)$ and $A^*=A$.

\begin{lemma}\label{iteration-lemma}
Let $u_\varep\in W^{1,2}(D_{3r})$ be a weak solution to $\mathcal{L}_\varep (u_\varep)=0$
in $D_{3r}$ and $u_\varep =0$ on $\Delta_{3r}$.
Suppose that for some $q=q_1>2$,
\begin{equation}\label{4.1.1}
\left\{\frac{1}{\rho^n}\int_{D_\rho} |\nabla u_\varep|^{q}\, dx \right\}^{1/q}
\le C_q
\left\{\frac{1}{\rho^n}\int_{D_{3\rho}} |\nabla u_\varep|^2\, dx \right\}^{1/2},
\end{equation}
for all $0<\rho\le r$. 
Then there exists $\delta>0$, depending only on $n$, $m$, $\mu$, $\lambda$,
$\tau$, $C_{q_1}$ and $M$, such that
the estimate (\ref{4.1.1}) holds for $2<q<2+\delta+\frac{q_1}{n}$
and $C_q=C(n,m,\mu,\lambda, \tau, q, C_{q_1}, M)$.
\end{lemma}

\begin{proof}
Let $d(x)=|x_n-\psi(x^\prime)|$ for $x=(x^\prime, x_n)$.
It follows from (\ref{interior-estimate}) and (\ref{4.1.1}) with $q=q_1$ that
\begin{equation}\label{4.1.2}
\aligned
|\nabla u_\varep (x)| & \le C \left\{ \frac{1}{[d(x)]^n}
\int_{B(x, cd(x))} |\nabla u_\varep (y)|^{q_1}\, dy \right\}^{1/q_1}\\
& \le C \left\{ \frac{\rho}{d(x)}\right\}^{\frac{n}{q_1}}
\left\{ \frac{1}{\rho^n} 
\int_{D_{3\rho}} |\nabla u_\varep (y)|^2\, dy\right\}^{1/2},
\endaligned
\end{equation}
for any $x\in D_{2\rho}$.
Since $A\in \Lambda(\mu, \lambda, \tau)$ and $A^*=A$, it follows from \cite{Kenig-Shen-2} that 
there exists $\delta>0$ such that the unique weak solution to the Dirichlet problem 
$\mathcal{L}_\varep (u_\varep)=0$ with boundary data in $W^{1,2+\delta}(\partial\Omega)$ 
in a Lipschitz domain $\Omega$ with connected boundary
satisfies
$\|(\nabla u_\varep)^*\|_{L^{2+\delta}(\partial\Omega)}
\le C \| \nabla_{tan} u_\varep\|_{L^{2+\delta}(\partial\Omega)}$.
Here $(\nabla u_\varep)^*$ denotes the nontangential maximal function of $\nabla u_\varep$.
By applying this estimate to $u_\varep$ on the Lipschitz domain $D_{t\rho}$ for
$t\in (3/2,2)$ and using an integration argument, one may obtain
\begin{equation}\label{4.1.3}
\int_{\Delta_\rho}
|(\nabla u_\varep)^*_\rho|^{2+\delta}\, d\sigma
\le \frac{C}{\rho} \int_{D_{2\rho}} |\nabla u_\varep|^{2+\delta}\, dx,
\end{equation}
where
\begin{equation}\label{4.1.4}
(\nabla u_\varep)^*_\rho (x^\prime, \psi(x^\prime))
=\sup \big\{ |\nabla u_\varep (x^\prime, x_n)|: \
(x^\prime, x_n )\in D_\rho \big\}.
\end{equation}
Let $q_0=2+\delta$. Note that, if $\delta$ is sufficiently small,
\begin{equation}\label{4.1.4.1}
\left\{\frac{1}{\rho^n}\int_{D_{2\rho}} |\nabla u_\varep|^{q_0}\, dx \right\}^{1/q_0}
\le C
\left\{\frac{1}{\rho^n}\int_{D_{3\rho}} |\nabla u_\varep|^2\, dx \right\}^{1/2}
\end{equation}
(see Remark \ref{Cacciopoli-remark}).
 Hence,
\begin{equation}\label{4.1.4.2}
\left\{ \frac{1}{\rho^{n-1}}\int_{\Delta_\rho}
|(\nabla u_\varep)^*_\rho|^{q_0}\, d\sigma\right\}^{1/q_0}
\le C \left\{ \frac{1}{\rho^n} \int_{D_{3\rho}} |\nabla u_\varep|^{2}\, dx\right\}^{1/2}.
\end{equation}

Now, using estimates (\ref{4.1.2}) and (\ref{4.1.4.2}), we see that
\begin{equation}\label{4.1.5}
\aligned
&\left\{ \frac{1}{\rho^n}\int_{D_\rho} |\nabla u_\varep|^{q}\, dx \right\}^{1/q}
=\left\{ \frac{1}{\rho^n}\int_{D_\rho} |\nabla u_\varep|^{q_0} |\nabla u_\varep|^{q-q_0}
\, dx \right\}^{1/q}\\
&\le C \left\{ \frac{1}{\rho^{n-1}}\int_{\Delta_\rho}
|(\nabla u_\varep)^*_\rho|^{q_0}\, d\sigma
\cdot \frac{1}{\rho}
\int_0^{c\rho} \left(\frac{\rho}{t}\right)^{\frac{n(q-q_0)}{q_1}}\, dt\right\}^{1/q}
\cdot \left\{\frac{1}{\rho^n}
\int_{D_{3\rho}} |\nabla u_\varep|^2\, dx\right\}^{\frac{q-q_0}{2q}}\\
&\le
C \left\{\frac{1}{\rho^n}
\int_{D_{3\rho}} |\nabla u_\varep|^2\, dx \right\}^{1/2},\nonumber
\endaligned
\end{equation}
if $0<n(q-q_0)<q_1$. Note that $n(q-q_0)<q_1$ is equivalent to
$q<2+\delta+\frac{q_1}{n}$.
This finishes the proof.
\end{proof}

\noindent{\bf Proof of Theorem \ref{main-theorem-2}. }
Let $u_\varep \in W^{1,2}(D_{3r})$ be a weak solution to
$\mathcal{L}_\varep (u_\varep)=0$ in $D_{3r}$ and $u_\varep =0$ on $\Delta_{3r}$.
It follows from the Cacciopoli's inequality that the weak reverse H\"older inequality
(\ref{4.1.1}) always holds for some $q_1>2$ under the ellipticity condition (\ref{quadratic-estimate})
(see Remark \ref{Cacciopoli-remark};
smoothness and periodicity conditions are not needed). Suppose that $q_1<\frac{2n}{n-1}$.
By Lemma \ref{iteration-lemma} estimate (\ref{4.1.1}) holds for some 
$q=q_2>2+\frac{\delta}{2}+\frac{q_1}{n}>q_1$. If $q_2<\frac{2n}{n-1}$, then the same argument would give
(\ref{4.1.1}) for $q=q_3>2+\frac{\delta}{2}+\frac{q_2}{n}>q_2$. Continuing this process,
we claim that there exists some $j$ such that estimate (\ref{4.1.1}) 
holds for some $q=q_j>\frac{2n}{n-1}$. For otherwise we would have a bounded
increasing sequence $\{q_j\}$
such that $q_{j+1}>2+\frac{\delta}{2}+\frac{q_j}{n}>q_j$.
Let $q$ be the limit of $\{q_j\}$. Then $q\ge 2+\frac{\delta}{2}+\frac{q}{n}$, which implies
that $q>p_n=\frac{2n}{n-1}$. It follows that $q_j>p_n$ if $j$ is sufficiently large.
Thus (\ref{4.1.1}) must hold for some $q=q_j>p_n$.
This completes the proof.
\qed

\section{ Proof of Theorem \ref{main-theorem-1}}
\setcounter{equation}{0}

Let $D_r$ and $\Delta_r$ be defined as in (\ref{definition-of-Delta}) with
$\|\nabla\psi\|_\infty\le M$.
In view of Remark \ref{duality-remark} and Theorem \ref{main-theorem-3}, as in the case of
Theorem \ref{main-theorem-2}, Theorem \ref{main-theorem-1}
is a consequence of the following.

\begin{theorem}\label{theorem-5.1}
Let $\mathcal{L}_\varep =-\text{\rm div} (A(x/\varep)\nabla)$
with $A\in \mathcal{M}(\mu, \lambda, \tau)$.
Suppose that $u_\varep \in W^{1,2}(D_{3r})$, $\mathcal{L}_\varep (u_\varep)=0$
in $D_{3r}$ and $u_\varep =0$ on $\Delta_{3r}$.
Then the estimate (\ref{4.1.1}) holds for $q=p_n=\frac{2n}{n-1}$
with a constant $C$ 
depending only $n$, $m$, $\mu$, $\lambda$, $\tau$ and $M$.
\end{theorem}

Since the nontangential maximal function estimates used in Lemma \ref{iteration-lemma}
are not available under the assumption $A\in \mathcal{M}(\mu, \lambda, \tau)$,
the proof of Theorem \ref{theorem-5.1} relies on
a compactness method motivated by \cite{AL-1987}.
In \cite{Shen-2008} the same approach was used to establish Theorem \ref{main-theorem-2}
in the case $m=1$. 

Throughout the rest of this section we will assume that
$A\in \mathcal{M}(\mu, \lambda, \tau)$.

\begin{lemma}\label{lemma-5.2}
Let $u_\varep\in W^{1,2}(D_{3r})$, $\mathcal{L}_\varep (u_\varep)=0$ in $D_{3r}$
and $u_\varep=0$ on $\Delta_{3r}$. Then for any $p>1$,
\begin{equation}\label{5.2.1}
\int_{0}^{cr}\!\!\!\int_{|x'|<r}|\nabla u_\varepsilon(x',\psi(x')+s)|^p \, dx'
ds\leq C_p \int_{0}^{2cr}\!\!\!\int_{|x'|<2r}\Big|\frac{u_\varepsilon(x',\psi(x')+s)}{s}\Big|^p \, 
dx'ds,
\end{equation}
where $c=(M+10n)$ and $C_p>0$ depends only on $n,p,\mu,\tau,\lambda$ and $M$.
\end{lemma}

\begin{proof} This follows from the interior estimate (\ref{interior-estimate}). 
The proof is similar to that of Lemma 3.2 in \cite{Shen-2008} and thus omitted.
\end{proof}

\begin{lemma}\label{lemma-5.3}
Let $L=-\text{\rm div}(\bar{A}\nabla)$, where $\bar{A}=(a_{ij}^{\alpha\beta})$ with
$1\le i, j\le n$ and $1\le \alpha, \beta\le m$ is a constant matrix satisfying
$\bar{A}^*=\bar{A}$ and the Legendre-Hadamard condition (\ref{Hardamard-condition}).
Suppose that $u_0\in W^{1,2}(D_{3/2})$, $L(u_0)=0$ in $D_{3/2}$ and
$u_0=0$ on $\Delta_{3/2}$.
Then
\begin{equation}\label{5.3.1}
\aligned
\int_0^t \int_{|x^\prime|< 1}
& |u_0(x^\prime, \psi(x^\prime) +s)|^{p_n}\, dx^\prime ds\\
&\le C_0\, t^{p_n+2\sigma}
\int_0^{3/2} \int_{|x^\prime|<\frac32} |u_0 (x^\prime, \psi(x^\prime) +s)|^{p_n}\,
dx^\prime ds
\endaligned
\end{equation}
for any $0<t<1$, where $C_0$ and $\sigma$ are positive constants depending only
on $n, m, \mu$ and $M$.
\end{lemma}

\begin{proof} Since $u_0=0$ on $\Delta_{3r}$, it follows by
the fundamental theorem of calculus that
\begin{equation}\label{5.3.2}
\int_0^t \int_{|x^\prime|<1}|u_0(x^\prime, \psi(x^\prime) +s)|^{p_n}\, dx^\prime ds
\le C\, t^{p_n} \int_0^t\int_{|x^\prime|<1} |\nabla u_0 (x^\prime, \psi(x^\prime)+s)|^{p_n}\,
dx^\prime ds.
\end{equation}
By H\"older's inequality the right hand side of (\ref{5.3.2}) is bounded by
$$
C\, t^{p_n +\frac{\delta}{p_n +\delta}}
\left\{ \int_0^1 \int_{|x^\prime|<1} |\nabla u_0 (x^\prime, \psi(x^\prime)+s)|^{p_n +\delta}\, 
dx^\prime ds\right\}^{\frac{p_n}{p_n+\delta}}.
$$
This, together with Lemma \ref{constant-coefficient-lemma}, implies that
\begin{equation}\label{5.3.3}
\aligned
\int_0^t \int_{|x^\prime|<1} & |u_0(x^\prime, \psi(x^\prime) +s)|^{p_n}\, dx^\prime ds\\
& \le C\, t^{p_n +\frac{\delta}{p_n+\delta}}
\left\{ \int_0^{\frac54} \int_{|x^\prime|<\frac54}
|\nabla u_0(x^\prime, \psi(x^\prime)+s)|^2\, dx^\prime ds\right\}^{p_n/2}.
\endaligned
\end{equation}
Let $2\sigma=\frac{\delta}{p_n+\delta}$.
Estimate (\ref{5.3.1}) now follows from (\ref{5.3.3}) by Caciopoli's and H\"older inequalities.
\end{proof}

Let $C_0$ and $\sigma$ be given by Lemma \ref{lemma-5.3}.
Choose $t_0\in (0,1/2)$ so small that $C_0 t_0^{\sigma} <(1/2)$.
Then $C_0t_0^{p_n+2\sigma}<(1/2)t_0^{p_n +\sigma}$.

\begin{lemma}\label{lemma-5.4}
There exists $\varepsilon_0>0$, depending only on $n$, $\mu$, $\lambda$, $\tau$ and $M$, 
such that for any $0<\varepsilon\leq \varepsilon_0$,
\begin{equation}\label{5.4}
\aligned
\int_0^{t_0}\!\!\int_{|x'|<1}&
|u_\varepsilon(x',\psi(x')+t)|^{p_n} \, dx'dt\\
&\leq t_0^{p_n+\sigma} \int_0^{3c}\!\!\int_{|x'|<3}|u_\varepsilon(x',\psi(x')+t)|^{p_n} 
\, dx'dt,
\endaligned
\end{equation}
where $c=(M+10n)$,
if $u_\varep\in W^{1,2}(D_3)$, $\mathcal{L}_\varep (u_\varep)=0$ in $D_3$ and
$u_\varep =0$ on $\Delta_3$.
\end{lemma}

\begin{proof}
We will prove the lemma by contradiction. 
For any $k\in \mathbb{N}$, denote
\begin{equation} 
\aligned
&D_r^k=\big\{(x^\prime,x_n): \ |x^\prime|<r \text{ and } 
 \psi_k(x^\prime)<x_n<\psi_k(x^\prime)+(M+10n)r\big\},\\
& \Delta_r^k=\big\{(x^\prime,x_n): \ |x^\prime|<r \text{ and }  x_n=\psi_k(x^\prime)\big\}, \nonumber
\endaligned
\end{equation}
where $\|\nabla \psi_k\|_\infty\leq M$ and $\psi_k(0)=0$.
Assume that there exist $\{\mathcal{L}^{(k)}\}$,
$\{\varepsilon_k\}$, $\{ \psi_k\}$ 
and $\{u_{\varepsilon_k}\}$ such that $\varepsilon_k\to 0$ as $k\to \infty$,
\begin{align}
\mathcal{L}_{\varepsilon_k}^{(k)}(u_{\varepsilon_k})
=-\text{\rm div} \Big(A^k \Big(\frac{x}{\varep_k}\Big)\nabla u_k\Big)=0
 \quad {\rm in}\  D_{3}^k ,  \quad  u_{\varepsilon_k}=0 \quad {\rm on} \  \Delta_{3}^k,
\end{align}
\begin{align}\label{5.4.1}
\int_0^{3c}\!\!\!\int_{|x'|<3}|u_{\varepsilon_k}(x^\prime,\psi_k(x^\prime)+t)|^{p_n} 
\, {d}x^\prime{d}t=1,
\end{align}
and
\begin{align}\label{5.4.2}
\int_0^{t_0}\!\!\!\int_{|x'|<1}|u_{\varepsilon_k}(x^\prime,
\psi_k (x^\prime)+t)|^{p_n} \, { d}x'{ d}t>t_0^{p_n+\sigma},
\end{align}
where  the coefficient matrices $A^k=\big(a_{ij}^{\alpha\beta, k} (y)\big)
 \in \mathcal{M}(\mu, \lambda, \tau)$.

Let
\begin{align}
b_{ij}^{\alpha\beta,k}=\int_{[0,1]^n}\Big[
a_{ij}^{\alpha\beta, k} +a_{i\ell}^{\alpha\gamma, k}
\frac{\partial}{\partial y_\ell} \big( \chi_j^{\gamma\beta,k}\big)\Big]\, dy,
\end{align}
where  $\chi^k (y)=(\chi^{\alpha\beta,k}_j (y))_{1\leq \alpha,\beta,j\leq n}$
 are correctors for $\mathcal{L}_\varep^{(k)}$.  
Note that $b_{ij}^{\alpha\beta,k}$ are bounded. Hence, by passing to a subsequence, we may
suppose that
\begin{align}
b_{ij}^{\alpha\beta}=\lim\limits_{k\to \infty}b_{ij}^{\alpha\beta,k}
\end{align}
exists for $1\le i,j,\alpha, \beta\leq n$.  
Since each $(b_{ij}^{\alpha\beta, k})\in \mathcal{M}(\widetilde{\mu}, \lambda, \tau)$
for some $\widetilde{\mu}$ depending only on $\mu$ 
(see e.g. \cite[p.202]{Cioranescu}), so does the matrix $(b_{ij}^{\alpha\beta})$.
We remark that $t_0$ and $\sigma$ should be chosen for this $\widetilde{\mu}$.

Since the sequence $\{\psi_k\}$ is equi-continuous on $\{x'\in \mathbb{R}^{n-1}:\, |x'|\leq 5\}$ 
and $\psi_k(0)=0$, by the  Ascoli-Arzela  theorem, 
we may assume that  $\psi_k$ converges uniformly to $\psi_0$ on $\{x':\, |x'|\leq 5\}$. 
We also have that $\|\nabla \psi_0\|_\infty\leq M$ and $\psi_0(0)=0$.
Let $v_k (x^\prime, t)=u_k (x^\prime, \psi_k (x^\prime)+t)$
and $ Q_r=\{(x',t):\, |x'|<r \text{ and } 0<t<cr\}$. 
Note that by Cacciopoli's inequality and (\ref{5.4.1}), $\{ v_k\}$ is 
uniformly bounded in $W^{1,2}(Q_{2})$.
Thu, by passing to a subsequence, we may assume that $v_k\to v_0$ weakly
in $W^{1,2}(Q_{2})$. 
Since $W^{1,2}(Q_2)$ is compactly embedded in $L^{p_n}(Q_2)$,  
we may assume that $v_k\to v_0$ strongly in $L^{p_n}(Q_2)$.
In view of (\ref{5.4.1}) and (\ref{5.4.2})
we obtain
\begin{equation}\label{5.4.3}
\aligned
& \int_0^2 \int_{|x^\prime|<2} |v_0(x^\prime, t)|^{p_n}\, dx^\prime dt \le 1,\\
& \int_0^{t_0} \int_{|x^\prime|<1} |v_0(x^\prime, t)|^{p_n}\, dx^\prime dt \ge t_0^{p_n+\sigma}.
\endaligned
\end{equation}

Now, let $w(x^\prime, x_n)=v_0(x^\prime, x_n-\psi_0 (x^\prime))$.
Then $w\in W^{1,2}(\widetilde{D}_2)$ and $w=0$ on $\widetilde{\Delta}_2$, where
$\widetilde{D}_r$ and $\widetilde{\Delta}_r$ are defined as in (\ref{definition-of-Delta}),
but with $\psi$ replaced by $\psi_0$.
Let $L=-\text{div}(\bar{A}\nabla)$, where $\bar{A}=(b_{ij}^{\alpha\beta})$.
It follows from the theory of homogenization that $L(w)=0$ in $D_2$ (see e.g. 
\cite[Lemma 2.1]{Kenig-Lin-Shen-1}).
In view of Lemma \ref{lemma-5.3} and (\ref{5.4.3}) we obtain
\begin{equation}\label{5.4.4}
\aligned
\int_0^{t_0}\int_{|x^\prime|<1} & |w (x^\prime, \psi_0 (x^\prime) +t)|^{p_n}\, dx^\prime dt\\
& \le C_0 t_0^{p_n +2\sigma}
\int_0^2 \int_{|x^\prime|<2} |w(x^\prime, \psi_0(x^\prime)+t)|^{p_n}\, dx^\prime dt\\
&\le (1/2) t_0^{p_n +\sigma},
\endaligned
\end{equation}
which contradicts the second inequality in (\ref{5.4.3}).
This completes the proof.
\end{proof}

\begin{lemma}\label{lemma-5.5}
Let $\varep_0>0$ be given by Lemma \ref{lemma-5.4}.
There exist positive constants $\delta$ 
and $C$, depending only on $n, \mu,\tau, \lambda$ and $M$,
such that for  $(\varepsilon/\varepsilon_0)<t<1$,
\begin{equation}\label{5.5.1}
\aligned
\int_0^t\!\!\!\int_{|x'|<1}&|u_\varepsilon(x',\psi(x')+s)|^{p_n}\, dx'ds\\
&\leq Ct^{p_n+\delta}\int_0^{3c}\!\!\!\int_{|x'|<3}
|u_\varepsilon(x',\psi(x')+s)|^{p_n}\, dx' ds,
\endaligned
\end{equation}
whenever $u_\varep\in W^{1,2}(D_{3})$, $\mathcal{L}_\varep (u_\varep)=0$ in $D_{3}$
and $u_\varep=0$ on $\Delta_{3}$.
\end{lemma}

\begin{proof}
Lemma \ref{lemma-5.5} follows from Lemma \ref{lemma-5.4} by a rescaling-iteration argument.
We refer the reader to \cite[pp.2294-2295]{Shen-2008} for details.
\end{proof}

Finally we are in a position to give the proof of Theorem \ref{theorem-5.1}.

\begin{proof}[Proof of Theorem \ref{theorem-5.1}]

By rescaling we may assume that $r=1$.
Let $\varepsilon_0$  be given by Lemma \ref{lemma-5.4}. 
If $\varepsilon\geq\epsilon_0/4$, estimate (\ref{4.1.1}) 
follows directly from Theorem \ref{boundary-Holder-theorem}.
Now we suppose that $\varepsilon<\varepsilon_0/4$.  
Observe that $v(x)=u_\varepsilon(\varepsilon x)$ is a weak solution of
${\mathcal L}_1 (v)=0$. Thus by Hardy's inequality and Theorem \ref{boundary-Holder-theorem},
\begin{equation}\label{5.6.1}
\aligned
\int_0^{\varepsilon /\epsilon_0}\!\!&
\int_{|x'|<\varepsilon /{\varepsilon_0}}
\Big|\frac{u_\varepsilon(x',\psi(x')+s)}{s}\Big|^{p_n}\, dx' ds\\
&\le C \int_0^{\varep/\varep_0} \int_{|x^\prime|<\varep/\varep_0}
|\nabla u_\varep (x^\prime, \psi(x^\prime) +s)|^{p_n}\, 
dx^\prime ds\\
&\leq \frac{C}{(\varepsilon )^{p_n}}\int_0^{c\varepsilon /\varepsilon_0}
\!\!\int_{|x'|<2\varepsilon /{\varepsilon_0}}
\big|u_\varepsilon(x',\psi(x')+s)\big|^{p_n}\, dx'{d}s.
\endaligned
\end{equation}
By covering $\Delta_1$ with surface balls of radius $\varepsilon /{\varepsilon_0}$, 
we can deduce from
(\ref{5.6.1}) that
\begin{equation}\label{5.6.2}
\aligned
\int_0^{\varepsilon /\varepsilon_0}\!\!&\int_{|x'|<1}
\Big|\frac{u_\varepsilon(x',\psi(x')+s)}{s}\Big|^{p_n}\, dx'ds\\
&\leq \frac{C}{\varepsilon^{p_n}}\int_0^{c\varepsilon /\varepsilon_0}\!\!
\int_{|x'|<2}\big|u_\varepsilon(x',\psi(x')+s)\big|^{p_n}\, { d}x'{ d}s\\
&\leq C\int_0^{3c}\!\!\int_{|x'|<3}
\big|u_\varepsilon(x',\psi(x')+s)\big|^{p_n}\, { d}x'{d}s,
\endaligned
\end{equation}
where we have used Lemma \ref{lemma-5.5} for the last inequality .

Next, we denote $f(x',s)=s^{-1}u_\varepsilon(x',\psi(x')+s)$ and write
\begin{equation}\label{5.6.3}
\aligned
&\int_0^c\!\!\!\int_{|x'|<1}|f(x',s)|^{p_n}\, { d}x'{d}s\\
&=\Big\{\int_0^{\varepsilon /{\varepsilon_0}}\!\!\int_{|x'|<1}
+\sum\limits_{j=1}^{j_0}\int_{2^{j-1}\varepsilon 
/{\epsilon_0}}^{2^j\varepsilon /{\varepsilon_0}}\!\!\int_{|x'|<1}
+\int_{2^{j_0}\varepsilon /{\varepsilon_0}}^c\!\!\int_{|x'|<1} 
\Big\}|f(x',s)|^{p_n}\, { d}x'{ d}s,
\endaligned
\end{equation}
where $2^{-j_0-1}\leq \varepsilon/{\varepsilon_0}\leq 2^{-j_0}$. 
The first term in the right hand side of (\ref{5.6.3})
 is handled by (\ref{5.6.2}).  Now we
apply (\ref{5.5.1}) to estimate the second term. This gives
\begin{equation}\label{5.6.4}
\aligned
\sum\limits_{j=1}^{j_0}&\int_{2^{j-1}\varepsilon 
/{\varepsilon_0}}^{2^j\varepsilon /{\varepsilon_0}}\!\!\int_{|x'|<1}|f(x',s)|^{p_n}
\, { d}x'{ d}s\\
&\leq C\sum\limits_{j=1}^{j_0} \Big({2^{j-1}\frac{\varepsilon}{{\varepsilon_0}}}\Big)^{-p_n} 
\Big( 2^j \frac{\varepsilon}{\varepsilon_0}\Big)^{p_n+\delta}
\int_0^{3c}\!\!\!\int_{|x'|<3}|u_\varepsilon(x',\psi(x')+s)|^{p_n}\, { d}x'{ d}s\\
&\leq C\int_0^{3c}\!\!\!\int_{|x'|<3}|u_\varepsilon(x',\psi(x')+s)|^{p_n}\, {d}x'{ d}s,
\endaligned
\end{equation}
where in the last inequality we has used  
$2^{-j_0-1}\leq \varepsilon/{\varepsilon_0}\leq 2^{-j_0}$.

Finally, the last term in (\ref{5.6.3}) is  controlled by
\begin{align}
C\int_0^{c}\!\!\!\int_{|x'|<1}|u_\varepsilon(x',\psi(x')+s)|^{p_n}\, {d}x'{ d}s.
\end{align}
Therefore, we have shown that
\begin{align}\label{5.6.6}
\int_0^1\!\!\!\int_{|x'|<1}\Big|\frac{u_\varepsilon(x',\psi(x')+s)}{s}\Big|^{p_n}
\, { d}x'{ d}s\leq  C\int_{D_{3}}|u_\varep(x)|^{p_n}\, { d}x.
\end{align}
In view of Lemma \ref{lemma-5.2} this implies that
\begin{equation}\label{5.6.7}
\int_{D_1} |\nabla u_\varep|^{p_n}\, dx 
 \le C \int_{D_3} |u_\varep|^{p_n}\, dx\\
\le C \left\{\int_{D_3} |\nabla u_\varep|^2\, dx \right\}^{p_n/2},
\end{equation}
where the last step follows from the Sobolev's inequality.
This completes the proof of Theorem \ref{theorem-5.1}.
\end{proof}

\medskip

\noindent{\bf Acknowledgment.}
Jun Geng and Zhongwei Shen were supported in part by NSF grant DMS-0855294.
Liang Song was supported in part by NNSF of China (Grant No.  11001276).

\bibliography{gss}

\bigskip

\begin{flushleft}
Jun Geng, 
Department of Mathematics, 
University of Kentucky,
Lexington, Kentucky 40506,
USA.
\quad
E-mail: jgeng@ms.uky.edu
\end{flushleft}

\bigskip

\begin{flushleft}
Zhongwei Shen (corresponding author), Department of Mathematics,
University of Kentucky,
Lexington, Kentucky 40506,
USA. Fax: 1-859-257-4078.

E-mail: zshen2@uky.edu
\end{flushleft}

\bigskip

\begin{flushleft}
 Liang Song,
Department of Mathematics,
Sun Yat-sen (Zhongshan) University,
Guangzhou, 510275,
P. R. China. \quad
E-mail: songl@mail.sysu.edu.cn
\end{flushleft}

\medskip

\noindent \today

\end{document}